\documentclass[12pt]{amsart}
\usepackage[margin=1in,centering]{geometry}
\usepackage[all]{xy}
\usepackage{amssymb}
\usepackage{graphicx,epsfig}
\usepackage[usenames,dvipsnames]{color}
\usepackage[implicit=true,pagebackref=true,%
  colorlinks=true,linkcolor=Blue,citecolor=PineGreen,urlcolor=BrickRed]{hyperref}
\newtheorem{thm}{Theorem}[section]
\newtheorem{cor}[thm]{Corollary}
\newtheorem{lem}[thm]{Lemma}
\newtheorem{prop}[thm]{Proposition}
\newtheorem{prob}{Problem}
\theoremstyle{definition}

\theoremstyle{remark}

\newtheorem{conj}{Conjecture}

\numberwithin{equation}{section}

\begin{document}

\title{On the  Gap Conjecture concerning Group Growth}
\author{Rostislav Grigorchuk}
\address{Department of Mathematics;
Mailstop 3368; Texas A\&M University; College Station, TX 77843-3368, USA }

\email{grigorch@math.tamu.edu}

\thanks{The author is partially supported  by the Simons Foundation and by NSF grant DMS - 1207699.}
%\subjclass[2010]{20F50,20F55,20E08} \keywords{growth of group, growth function, growth series, Milnor
%problem, polynomial growth, intermediate growth, exponential growth, uniformly exponential growth,  gap
%conjecture, amenable group, self-similar group, branch group, group generated by finite automaton}

%\date{}%
%\dedicatory{Dedicated to John Milnor on the occasion of his  80th birthday.}
%\commby{}%
% ----------------------------------------------------------------
\begin{abstract}   We discuss some new results concerning  Gap Conjecture on group growth
 and present a reduction of
it   (and its $*$-version) to several special classes of groups. Namely we show that its validity for the
classes of simple groups and residually finite groups will imply the Gap Conjecture in full generality. A
similar type reduction holds if the Conjecture is valid for residually polycyclic groups and  just-infinite
groups. The cases of residually solvable groups and right orderable groups are considered as well.
\end{abstract}
\maketitle
% ----------------------------------------------------------------

\section{Introduction}

Growth functions of finitely generated groups were introduced by  A.S.~Schvarz \cite{schvarts:55} and
independently by J.~Milnor \cite{milnor:note68},  and remain  popular subject of geometric group theory.
Growth of a finitely generated group can be polynomial, exponential or intermediate between polynomial and
exponential. The class of groups of polynomial growth coincides with the class of virtually nilpotent groups
as was conjectured by Milnor  and confirmed by M.~Gromov \cite{gromov:poly_growth}.   Milnor's problem on the
existence of groups of intermediate growth was solved by the author in
\cite{grigorch:degrees,grigorch:degrees85}, where for any prime $p$  an uncountable family of 2-generated
torsion $p$-groups $\mathcal G_{\omega}^{(p)}$    with different types of  intermediate  growth was
constructed. Here $\omega$ is a parameter of construction taking values in the space of infinite sequences
over the alphabet on $p+1$ letters. All groups $\mathcal G_{\omega}^{(p)}$ satisfy the following lower bound
on growth function
\begin{equation} \label{gapineq}\gamma_{\mathcal G_{\omega}}(n)\succeq e^{\sqrt n},
\end{equation}
where $\gamma_G(n)$ denotes the growth function of a group $G$ and $\succeq$ is a natural comparison of
growth functions  (see the next section for definition).  The inequality (\ref{gapineq}) just indicates that
growth of a group is not less than the growth of the function $e^{\sqrt n}$.

All groups from families  $\mathcal G_{\omega}^{(p)}$   are residually finite-$p$ groups (i.e. are
approximated by finite $p$-groups). In \cite{grigorch:hilbert} the author proved that the lower bound
(\ref{gapineq}) is universal for all residually finite-$p$ groups and this fact has a straightforward
generalization to residually nilpotent groups, as it is indicated in \cite{lubotzky_mann:polyn91}.

The paper \cite{grigorch:degrees85} also contains an example of a torsion free group of intermediate growth,
which happened to be right orderable group, as was shown in \cite{grigorch_machi:93}.   For this group the
lower bound (\ref{gapineq})  also holds.

In the ICM Kyoto paper~\cite{grigorch:ICM90} the author raised a question if the function  $e^{\sqrt n}$
gives a universal lower bound for all groups of intermediate growth. Moreover, later he conjectured that
indeed this is the case.  The corresponding conjecture is now called  the \emph{Gap Conjecture} on group
growth. In this note we collect known  facts related to the Conjecture  and present some new results.  A
recent paper \cite{grigorch:milnor11}  gives further information about the history and developments around
the notion of growth in group theory.

The first part of the note is introductory.  The second part begins with the case of  residually solvable
groups where basically we present some of  results of J.S.~Wilson
from~\cite{wilson:growthsolv05,wilson:gap11} and a consequence from them. Then we consider the case of right
orderable groups, and the final part contains two reductions of the Conjecture (and its $*$-version) to the
classes of residually finite groups and simple groups (Theorem \ref{resid}), and to the class of
just-infinite groups, modulo its correctness  for residually polycyclic groups (Theorem \ref{just-inf}).
% At the end we discuss gap type
%conjectures associated with other asymptotic characteristics of groups:  probabilities $P(n)$ of return for a
%random walk,   F\"{o}lner function $\mathcal F (n)$ and spectral function $\mathcal N (\lambda)$.  We  show
%that basically these conjectures are equivalent to the  Gap Conjecture on group growth and this implies their
%validity for certain classes of groups as well as the reduction type results. We emphasize  that the gap type
%conjectures related to these asymptotic characteristics  were formulated by P.~Pansu and author around 2000
% but the corresponding note was not written in full at that
%time.

 \section{Ackonwledgment} This work was
completed during  visit of the author to the Institute Mittag-Leffler (Djursholm, Sweden) associated with the
program ``Geometric and Analytic Aspects of Group Theory''.   The author acknowledges organizers of this
program.
 Also the author  would like to thank A.~Mann for indication  of the article
\cite{bajorska_maked:note07}, and   I.~Bondarenko and E.~Zelmanov for numerous valuable remarks concerning
the first draft of this note.

\section{Preliminary facts}\label{prelimin}

Let $G$ be a finitely generated group with a system of generators $A=\{a_1,a_2,\dots,a_m\}$ (throughout the
paper we consider only infinite finitely generated groups and only finite systems of generators). The
\emph{length} $|g|=|g|_A$ of an element $g\in G$ with respect to $A$ is the length $n$ of the shortest
presentation of $g$ in the form
\[g=a^{\pm 1}_{i_1}a^{\pm 1}_{i_2}\dots a^{\pm 1}_{i_n},\]
where $a_{i_j}$ are elements in $A$. It depends on the set of generators, but for any two systems of
generators $A$ and $B$ there is a constant $ C \in \mathbb{N}$ such that the inequalities
\begin{equation} \label{length} |g|_A \leq C|g|_B,  \qquad |g|_B \leq C|g|_A.
\end{equation}
hold.

The \emph{growth function} of a  group $G$ with respect to  the generating set $A$ is the function
\[\gamma_G^A(n)=\bigl|\{g\in G : |g|_A\leq n\}\bigr|,\]
where $|E|$ denotes the cardinality of a set $E$, and $n$ is a natural number.

If $\Gamma=\Gamma(G,A)$ is the Cayley graph of a group $G$ with respect to the generating set $A$, then $|g|$
is the combinatorial distance between vertices $g$ and  $e$ (the identity element in $ G$), and
$\gamma_G^A(n)$ counts the  number of vertices at combinatorial distance $\leq n$
 from  $e$ (i.e., it counts the number of elements in the
ball of radius $n$ with  center at the identity element).

It follows from (\ref{length}) that  growth functions $\gamma_G^A(n), \gamma_G^B(n)$ satisfy  the
inequalities
\begin{equation}\label{growth} \gamma_G^A(n) \leq  \gamma_G^B(Cn), \qquad
\gamma_G^B(n) \leq \gamma_G^A(Cn).
\end{equation}

The dependence of  the growth function on   generating set is  inconvenience and it is customary to avoid it
by using the following trick.  Two functions on the naturals $\gamma_1$ and $\gamma_2$ are called
\emph{equivalent} (written $\gamma_1\sim \gamma_2$)
 if there is a constant
$C\in \mathbb{N}$ such that $ \gamma_1(n) \leq  C\gamma_2(Cn)$, $\gamma_2(n) \leq  C\gamma_1(Cn)$ for all
$n\geq 1.$ Then according to (\ref{growth}), the growth functions constructed with respect to two different
systems of generators are equivalent. The class of equivalence $[\gamma_G^A]$ of  growth function  is called
 \emph{degree of growth}, or   \emph{rate of growth} of  $G$. It is an invariant  not
only up to   isomorphism but also  up to  weaker equivalence relation called \emph{quasi-isometry}
\cite{harpe:topics}.

\smallskip

We will also consider a preoder $\preceq  $ on the set of growth functions:
\begin{equation}\label{preoder}\gamma_1(n) \preceq \gamma_2(n)
\end{equation}
if there is an integer $C>1$ such that $\gamma_1(n) \leq \gamma_2(Cn)$ for all $n\geq 1$. This converts the
set $\mathcal{W}$ of growth degrees of finitely generated groups  into a partially ordered set. The notation
$\prec$ will be used in this article  to indicate a strict inequality.

\bigskip
Let us remind some basic facts about  growth rates that will be used in the paper.
\begin{itemize}
\item  The power functions $n^{\alpha}$ belong to different equivalence classes for different $\alpha
    \geq 0$.
\item  The polynomial function $P_d(n)=c_dn^d+ \dots +c_1n+c_0,$ where $c_d \neq  0$ is equivalent to the
    power function $n^d$.
\item All exponential  functions $\lambda^n, \lambda > 1$ are equivalent and belong to the class $[e^n]$.
\item All functions of \emph{intermediate type}  $e^{n^{\alpha}}, 0< \alpha <1$ belong to different
    equivalence classes.

\end{itemize}
This is not a complete list of rates of growth that a group may have.   Much more is provided in
\cite{grigorch:degrees} and \cite{bartholdi_ersch:given(2)11}.
%Moreover we will also consider functions of positive argument $\lambda$  approaching zero when $\lambda \to
%0+$ and apply a similar comparison techniques that would  allow us to speak about the rate of decay when
%$\lambda \to 0+$.

It is easy to see that  growth of a group coincides with the growth of a subgroup of finite index, and that
growth of a group is not smaller than the growth of a finitely generated subgroup or of a factor group. Since
a group with $m$ generators can be presented as a quotient group of a free group of rank $m$, the growth of a
finitely generated group cannot be faster than exponential (i.e., it can not be superexponential). Therefore
we can split the growth types into three classes:
\begin{itemize}
\item \emph{Polynomial} growth. A group $G$ has  \emph{polynomial} growth if there are constants $C>0$
    and $d> 0$ such that $\gamma(n) < Cn^d$ for all $ n\geq 1.$  Minimal  $d$ with this property is
    called the degree of polynomial growth.

\item \emph{Intermediate} growth.   A group $G$ has  \emph{intermediate}  growth if $\gamma(n)$ grows
    faster than  any polynomial but slower than  any exponent function $\lambda^n, \lambda >1$ (i.e.
    $\gamma(n) \prec e^n$).
\item \emph{Exponential growth}.  A group $G$ has  \emph{exponential} growth if  $\gamma(n)$ is
    equivalent to $e^n$.
\end{itemize}

 The  question on the existence of groups of intermediate growth was raised in
     1968 by Milnor \cite{milnor:problem}. For many classes of groups (for instance for linear groups by Tits
     alternative \cite{tits:alternative72}, or for solvable groups by the results of Milnor
     \cite{milnor:solv68} and Wolf \cite{wolf:growth68}) intermediate growth is impossible.
      Milnor's question was answered by  author in 1983 \cite{grigorch:milnor,grigorch:degrees,grigorch:habil}, where it
     was shown that there are uncountably many 2-generated torsion  groups of intermediate growth. Moreover,
     it was shown in \cite{grigorch:degrees,grigorch:degrees85,grigorch:habil} that for any prime $p$ a  partially ordered set $\mathcal W_p$ of growth
     degrees of finitely generated torsion p-groups  contains  uncountable chain
      and contains  uncountable anti-chain. The immediate
     consequence
     of this result is the existence
     of  uncountably many  quasi-isometry equivalence classes of finitely generated groups (in fact 2-generated groups) \cite{grigorch:degrees}.

Below we will   use  several times the following lemma   (\cite[page 59]{gromov:poly_growth}).

\begin{lem} [Splitting lemma] \label{splitting} Let   $G$ be a finitely generated group of polynomial growth of degree $d$ and $H\lhd G$ be a
normal subgroup with quotient $G/H$ being an infinite cyclic group.  Then $H$ has  polynomial growth of
degree $\leq d-1$.

\end{lem}

\section{Gap Conjecture and its modifications}\label{conjecture}

 We will say that a group is \emph{virtually nilpotent} (virtually solvable) if it contains
nilpotent (solvable) subgroup of finite index.  It was observed around 1968 by Milnor, Wolf, Hartly and
Guivarc'h  that a nilpotent group has polynomial growth and hence  a virtually nilpotent group also has
polynomial growth. In his remarkable paper \cite{gromov:poly_growth}, Gromov established the converse.

\begin{thm} (Gromov 1981) If a finitely generated group  $G$ has polynomial growth, then $G$ contains
a nilpotent subgroup of finite index.
\end{thm}

In fact Gromov obtained  stronger result about  polynomial growth.
\begin{thm}\label{gromoveffect} For any positive integers $d$ and $k$, there exist positive integers $R, N$ and $q$ with the
following property. If a group $G$ with a fixed system of generators satisfies the inequality $\gamma(n)\leq
kn^d$ for $n=1,2,\dots,R$ then $G$ contains a nilpotent subgroup $H$ of index at most $q$ and whose degree of
nilpotence is at most $N$.
\end{thm}

 The above theorem  implies  existence of a function $\upsilon$ growing faster than any
polynomial and such that if $\gamma_G\prec \upsilon$, then  growth of  $G$ is polynomial.

Indeed, taking a sequence $\{k_i,d_i\}_{i=1}^{\infty}$ with $k_i \to \infty $ and $d_i \to \infty$ when $i
\to \infty$ and the corresponding sequence $\{R_i\}_{i=1}^{\infty}$, whose existence follows from
Theorem~\ref{gromoveffect}, one can build a function $ \upsilon(n)$ which coincides with the polynomial
$k_in^{d_i}$ on the interval $[R_{i-1}+1,R_i]$  and separates polynomial growth from intermediate.
 Therefore
there is a \emph{Gap} in the scale of rates of growth of finitely generated groups and a big problem is to
find the optimal function (or at least to provide good lower and upper bounds for it) which separates
polynomial growth  from intermediate. The best known result in this direction is the function $ n^{(\log \log
n)^c}$ ($c$ some positive constant) which appeared recently in the paper of Shalom and Tao \cite[Corollary
8.6] {shalom_tao:polynom10}.

\smallskip

The lower bound of the type $e^{\sqrt n}$  for all groups $\mathcal G_{\omega}^{(p)}$ of intermediate growth
established in \cite{grigorch:milnor,grigorch:degrees,grigorch:degrees85,grigorch:habil}
 allowed the author to guess that
equivalence class  of function $e^{\sqrt n}$ could be a good candidate for a ``border'' between polynomial
and exponential growth.  This guess  was further strengthened  in 1988 when the author obtained the result
published in~\cite{grigorch:hilbert} (see Theorem~\ref{gap1}). For the first time the Gap Conjecture was
formulated in the form of  a  question  in 1991 (see \cite{grigorch:ICM90}).

\begin{conj}  \label{c:gap1}(Gap Conjecture) If the growth function $\gamma_G(n)$ of a finitely generated group $G$ is strictly bounded from
above by   $e^{\sqrt{n}}$ (i.e. if $\gamma_G(n) \prec e^{\sqrt n}$), then  growth of $G$ is polynomial.
\end{conj}

The question  of independent interest is whether there is a group, or more generally a cancellative
semigroup, with growth equivalent to $e^{\sqrt{n}}$  (for the role of cancellative semigroups in growth
business see \cite{grigorch:semigroups_with_cancel88}).

In \cite{grigorch:milnor11} the author formulated a number of conjectures relevant to the main Conjecture
discussed there and in this note. Let us recall some of them as they will play some role in what follow.

\begin{conj}  \label{c:gap2}(Gap Conjecture with parameter $\beta$, $0<\beta<1$). If the growth function
$\gamma_G(n)$ of a finitely generated group $G$ is strictly bounded from above by $e^{n^\beta}$ (i.e. if
$\gamma(n) \prec e^{n^\beta}$) then the growth of $G$ is polynomial.
\end{conj}

Thus the Gap Conjecture with parameter $1/2$ is just the Gap Conjecture~\ref{c:gap1}. If  $\beta<1/2$ then
the Gap Conjecture with parameter $\beta$ is weaker than the Gap Conjecture,  and if  $\beta>1/2$ then it is
stronger than the  Gap Conjecture.

\begin{conj} \label{c:gap3} (Weak Gap Conjecture).  There is a $\beta, 0 < \beta <1$ such that if  $\gamma_G(n) \prec
e^{n^\beta}$ then the Gap Conjecture with parameter $\beta$ holds.

\end{conj}

The gap type conjectures can be formulated for other asymptotic characteristics of groups like return
probabilities $P^{(n)}_{e,e}$ ($e$ denotes the identity element) for a non degenerate random walk on a group,
F\"{o}lner function $\mathcal F(n)$, or spectral density  $\mathcal N (\lambda)$.  There is a close relation
between them and the Gap Conjecture on growth, which was mentioned in~\cite{grigorch:milnor11}.   When
writing this note the author realized that to understand better the relation between different forms of the
gap type conjectures it is useful to consider in parallel to the conjecture \ref{c:gap2}  (which we will
denote $G(\beta)$) a  stronger version of it, which we will denote $G^*(\beta)$:

\begin{conj} [ Conjecture  $G^*(\beta)$] If a group $G$ is not virtually nilpotent then $\gamma_G(n)\succeq e^{n^\beta}$.

\end{conj}

It is obvious that $G^*(\beta)$ implies  $G(\beta)$ but the opposite is not clear.  This is related to the
fact that there are groups with incomparable growths~\cite{grigorch:degrees} as   the set $\mathcal W$ of
rates of growth of finitely generated groups is not linear ordered.   The motivation for introducing a
$*$-version  of the Gap Conjecture will be more clear when a second note  \cite{grigorch:gap2} of the author
 is submitted to the arXiv.

\section{Growth and elementary amenable groups} \label{amen}

Amenable groups were introduced by von Neumann  in 1929~\cite{vonNeumann:1929}. Now they play extremely
important role in many branches of mathematics. Let $AG$ denote the class of amenable groups. By a theorem of
Adelson-Velskii \cite{adelson:banach57}, each finitely generated group of subexponential growth belongs to
the class $AG$. This class contains finite groups and commutative groups and
 is closed under the following operations:
\begin{enumerate}
\item taking   a \emph{subgroup},
\item taking  a  \emph{quotient group},
\item  \emph{extensions},
\item  \emph{ unions} (i.e. if for some net $\{\alpha\}, G_\alpha \in AG \   \text{and} \ G_\alpha\subset
    G_{\beta}\ \text{if}\ \alpha<\beta$ then $\cup_{\alpha}G_\alpha\in AG$).
\end{enumerate}

\medskip

Let $EG$ be the class of \emph{elementary} amenable groups i.e., the smallest class of groups containing
finite groups, commutative groups which is closed with respect to the  operations (1)-(4). For instance,
virtually nilpotent and, more generally, virtually solvable groups belong to the class $EG$. This concept
defined by M.~Day in \cite{day:amenable} got  further development in the article \cite{chou:eg} of Chou who
suggested the following approach to study of elementary amenable groups.

 For each ordinal
$\alpha$ define a subclass $EG_{\alpha}$ of $EG$ in the following way. $EG_0$ consists of finite groups and
commutative groups. If
 $\alpha$ is a limit ordinal then
\[ EG_{\alpha}=\bigcup_{\beta \preceq \alpha} EG_{\beta}.\]
Further, $EG_{\alpha+1}$  is defined as as the class of groups which are extensions of groups from  set
$EG_{\alpha}$ by groups from  the same set. It is known (and easy to check) that each of the classes
$EG_{\alpha}$   is closed with respect to the operations (1) and (2) \cite{chou:eg}. By the \emph{elementary
complexity} of a group $G\in EG$ we call the smallest $\alpha$ such that $G \in EG_{\alpha}$.

It was shown in \cite{chou:eg}  that  class $EG$ does not contain groups of intermediate growth, groups of
Burnside type (i.e. finitely generated infinite torsion groups), and finitely generated infinite simple
groups. A further study of elementary groups and its generalizations was done by D.~Osin
\cite{osin:elementary04}.

A larger class $SG$ of subexponentially amenable groups was (implicitly) introduced in
\cite{freedman_teich:subexp95}, and explicitly in \cite{grigorch:example}, and studied in
\cite{harpe_cg:paradoxical} and other papers.

A useful fact about groups of intermediate growth which we will use is due to S.~Rosset \cite{rosset:76}.

\begin{thm}\label{rosset} If $G$ is a finitely generated group which does not grow exponentially and
 $H$ is a normal subgroup such that $G/H$ is solvable, then $H$ is finitely generated.
\end{thm}

We propose the following  generalization of this result.

 \begin{thm}\label{rosset1} Let $G$ be a finitely generated group with no
free subsemigroup on two generators and let the quotient $G/N$ be an elementary amenable group.  Then the
kernel $N$ is a finitely generated group.
\end{thm}

The latter two statements and the chain of further statements of the same spirit that appeared in the
literature were initiated by the following lemma of Milnor \cite{milnor:solv68}: if $G$ is a finitely
generated group with subexponential growth, and if $x, y \in G$,
 then the group generated by the set of conjugates $y, xyx^{-1}, x^2yx^{-2}, \dots $ is finitely generated.

\begin{proof}

 For the proof of the Theorem~\ref{rosset1} we will apply induction on elementary complexity $\alpha$ of the quotient group $H=G/N$.  If complexity is 0
then the  group is either finite or abelian.  In the first case $N$ is finitely generated for obvious reason.
In the second case we apply the following  statements from  the paper of P.~Longobardi and A.~Rhemtulla
\cite[Lemmas 1,2]{rhemtulla_long:freesemig95}.

\begin{lem}If G has no free subsemigroups, then for all $a, b \in G$ the subgroup $ \langle a^{b^n}, n\in \mathbb Z \rangle$
 is finitely generated.
\end{lem}
\begin{lem}Let $G$ be a finitely generated group. If $N \trianglelefteq G, G/N$ is cyclic, and  $\langle a^{b^n}, n\in \mathbb Z \rangle$
 is finitely generated for all
$a, b \in G$, then $N$ is finitely generated.
\end{lem}

Assume that the statement of the theorem is correct for quotients $H=G/H$ with complexity $\alpha \leq
\beta-1$ for some ordinal $\beta, \beta \geq 1$.    The group $H$, being finitely generated, allows a short
exact sequence

\[\{1\}\rightarrow A \rightarrow H \rightarrow B \rightarrow \{1\},\]
where $A,B \in EG_{\beta -1}$. Let $\varphi: G\rightarrow G/N$ be the canonical homomorphism and $M=
\varphi^{-1}(A)$. Then  $M$ is a normal subgroup in $G$ and $G/M\simeq G/N/M/N\simeq H/A \simeq B$. By the
inductive assumption $M$ is finitely generated and has no free subsemigroup on two generators.  As $M/N\simeq
A$, again by induction, $N$ is finitely generated and we are done.

\end{proof}

We will discuss  just-infinite groups in detail  in the last section.  But let us prove now a preliminary
result which will be used later.  Recall that a group is called just-infinite if it is infinite, but every
proper quotient is finite (i.e. every nontrivial normal subgroup is of finite index). A group  $G$ is called
hereditary just-infinite if it is residually finite and every subgroup $H < G$ of finite index  (including
$G$ itself) is just infinite.  Observe that a subgroup of finite index of a hereditary just-infinite group is
hereditary just-infinite.

We learned the following result from Y. de Cournulier.  A proof is provided here as there  is no one in the
literature.
\begin{thm} \label{element} Let $G$ be a finitely generated hereditary just-infinite group, and suppose that $G$  belongs to the class $EG$ of elementary amenable
groups.  Then $G$ is isomorphic either to the infinite cyclic group $\mathbb Z$  or to the infinite dihedral
group $D_{\infty}$.
\end{thm}

\begin{proof}  If $G \in EG_0$ then $G$ is abelian and hence $G\simeq \mathbb Z$.  Assume that the statement is
correct for all groups from classes $EG_{\alpha}, \alpha <  \beta$ for some ordinal $\beta$. Let us prove it
for $\beta $.  Assume  $G \in EG_{\beta}$ and $\beta$ is smallest with this property.   $\beta$ can not be a
limit ordinal because $G$ is finitely generated. Therefore $G$ is the extension of a group $A$ by a group
$B=G/A$, where $A,B \in EG_{\beta -1}$.  In fact $B$ is a finite group (as $G$ is just-infinite). As a
subgroup of finite index in a hereditary just-infinite group,  $A$ is hereditary just-infinite and moreover
finitely generated (as a subgroup of finite index in a finitely generated group). By inductive assumption $A$
is isomorphic either to the infinite cyclic group $\mathbb Z$  or to the infinite dihedral group
$D_{\infty}$. In particular $G$ has a normal subgroup  $H$ of finite index isomorphic to $\mathbb Z$.

Let $G$ act on $H$ by conjugation.  Then we get a homomorphism $\psi:G\rightarrow Aut(H) \simeq \mathbb Z
_2$. If $\psi(G)=\{1\}$, then $H$ is a central subgroup. It is a standard fact in group theory (see for
instance \cite[Proposition 2.4.4] {karpilovsky:schur})  that if there is a central subgroup of finite index
in $G$ then the commutator subgroup $G'$ is finite. But as $G$ is just-infinite, $G'=\{1\}$ and so $G$ is
abelian, hence $G\simeq \mathbb Z$ in this case.

If $\psi(G)=Aut(H)$ then $N= ker \psi $  is a centralizer $C_G(H)$ of $H$ in $G$. Subgroup $N$ has index 2 in
$G$, is just-infinite and hence by the same reason as above $N'=\{1\}$, so $N$ is abelian. Being finitely
generated and just infinite implies $N\simeq \mathbb Z$.

Let $x\in G, x \notin N$. The element  $x$ acts on $N$ by conjugation mapping each  element to its inverse.
In particular, $x^{-1}(x^2)x=x^{-2}$, so $(x^2)^2=1$. But $x^2 \in N$. Since $N$ is torsion free $x^2=1$.
Therefore
\[G=\langle x,N\rangle=\langle x,y:x^2=1, x^{-1}yx=y^{-1}\rangle \simeq D_{\infty},\]
where $y$ is a generator of $N$.

\end{proof}

\section{Gap Conjecture for residually solvable groups}

Recall that a group $G$ is said to be a residually finite-$p$ group (sometimes also called residually finite
$p$-group) if it is approximated by finite $p$-groups, i.e., for any $g\in G$ there is a finite $p$-group $H$
and a homomorphism $\phi:G\rightarrow H$ with $\phi(g)\neq 1$. This class is, of course, smaller than the
class of residually finite groups, but it is pretty large.  For instance, Golod-Shafarevich groups,
$p$-groups $\mathcal G_{\omega}$ from \cite{grigorch:degrees,grigorch:degrees85}, and many other groups
belong to this class.

\begin{thm} \label{gap1} (\cite{grigorch:hilbert})  Let $G$ be a finitely generated residually finite-$p$    group.
If $\gamma_G(n)\prec e^{\sqrt{n}}$ then $G$ has polynomial growth.
\end{thm}

As was established by the author in a discussion with A.~Lubotzky and A.~Mann during the conference on
profinite groups in Oberwolfach in 1990,
  the same arguments as given in \cite{grigorch:degrees} combined with the following  lemma
 from \cite{lubotzky_mann:polyn91}

\begin{lem} [Lemma 1.7, \cite{lubotzky_mann:polyn91}] Let $G$ be a finitely generated residually nilpotent group. Assume that for every prime $p$ the
pro-$p$-closure $G_{\hat p}$ of $G$  is $p$-adic analytic. Then $G$ is linear.

\end{lem}
\noindent allows one to prove a stronger version of the above theorem (see  the Remark after Theorem~1.8 in
\cite{lubotzky_mann:polyn91}):

\begin{thm}\label{lubmann}
Let $G$ be a residually nilpotent finitely generated group. If $\gamma_G(n)\prec e^{\sqrt{n}}$ then $G$ has
polynomial growth.
\end{thm}

To be linear means to be isomorphic to a subgroup of the linear group $GL_n(\mathbb K)$ for some field
$\mathbb K$.   By Tits alternative \cite{tits:alternative72} every finitely generated linear group either
contains a free subgroup on two generators or is virtually solvable. Hence the above lemma immediately
reduces Theorem~\ref{lubmann} to  Theorem~\ref{gap1}.

The latter two theorems (where the first one  is the corresponding statement from \cite{grigorch:hilbert}
while the second one is  a corrected form of what is stated in  Remark on page 527
in~\cite{lubotzky_mann:polyn91}) show that  Gap Conjecture $G(1/2)$ holds for the class of residually
finite-$p$ groups and more generally for the class of residually nilpotent groups.  In fact, arguments
provided in \cite{grigorch:hilbert,lubotzky_mann:polyn91} allow to prove  stronger conjecture $G^*(1/2)$ for
these classes of groups.

Let  $p$ be a prime and $a_n^{(p)}$ be the $n$-th coefficient of  the power series given by

\[\sum_{n=0}^{\infty}a_n^{(p)}z^n=\prod_{n=1}^{\infty}\frac{1-z^{pn}}{1-z^n}.\]
Then   the lower bound $a_n^{(p)} \succeq e^{\sqrt{n}}$ holds. Moreover if a group $G$ is a residually
finite-$p$ group and is not virtually nilpotent then  for any  system of generators $A$
\[\gamma_G^A(n) \geq a_n^{(p)}, n=1,2,\dots\]
(see the relation (23) and Lemma 8 in~\cite{grigorch:hilbert}). Observe that the latter statement  is valid
not only in the case when $A$   is a system of elements that generate $G$ as a group but even in a more
general case when $A$ is a generating set for the group $G$ considered as a semigroup.  In fact,   growth
function of any group is bounded from below by a  sequence of coefficients of Hilbert-Poincar\'{e} series of
the universal $p$-enveloping algebra of the restricted Lie $p$-algebra associated with the group using the
factors of the lower $p$-central series~\cite{grigorch:hilbert}.

  Theorem~1.8 from  \cite{lubotzky_mann:polyn91} contains an
interesting approach  to polynomial growth type theorems in the case of residually nilpotent groups.
Moreover, as is mentioned in \cite{lubotzky_mann:polyn91} in the remark after the theorem,  the proof
provided  there yields the same conclusion  under a weaker assumption:  $\gamma_G(n)\prec 2^{2^{\sqrt{\log_2
n}}}$.

Surprisingly, in his first paper  on the gap type problem \cite{wilson:growthsolv05} Wilson used a similar
upper bound   $\gamma_G(n)\prec e^{e^{(1/2) \sqrt {\ln n}}}$ to measure  size  of a gap for residually
solvable groups. Wilson's approach is quite different from those that were used before and is based on
exploring self-centralizing chief factors in finite solvable groups.

Recall  that a chief factor of a group $G$ is a (nontrivial) minimal normal subgroup of some quotient $G/N$,
and that $L/M$ is a self-centralizing chief factor of a group $G$ if $M$ is normal in $G$, $L/M$ is a minimal
normal subgroup of $G/M$, and $L/M=C_{G/M}(L/M)$. One of the results in  \cite{wilson:growthsolv05}  is

\begin{thm}[Wilson] \label{wilson2}Let $G$ be a residually solvable group of subexponential growth whose finite
self-centralizing chief factors all have rank at most $k$. Then $G$ has a residually nilpotent normal
subgroup whose index is finite and bounded in terms of $k$ and $\gamma_G(n)$.

If, in addition $\gamma_G(n)\prec e^{\sqrt n}$, then $G$ has a nilpotent normal subgroup whose index is
finite and bounded in terms of $k$ and $\gamma_G(n)$.

\end{thm}

The proof of this result is based on the following lemma the proof of which uses ultraproducts.

\begin{lem} [Lemma 2.1, \cite{wilson:growthsolv05}] Let $k$ be a positive integer and $\alpha:\mathbb N \rightarrow \mathbb R_{+}$ a
function such that $\alpha(n)/n \to 0 $ as $n  \to \infty$. Suppose that $G$ is a finite solvable group
having (i) a self-centralizing minimal normal subgroup $V$ of rank at most $k$ and (ii) a generating set $A$
such that $\gamma_G^A(n)\leq e^{\alpha(n)}$ for all $n$. Then $|G/V|$ is bounded in terms of $k$ and $\alpha$
alone.
\end{lem}

 One of the almost  immediate corollaries of the technique developed in  \cite{wilson:growthsolv05} are the
 facts  stated below in Theorems \ref{polyciclic} and \ref{wilson6}.

Recall that a group is called supersolvable if it has a finite  normal descending  chain of subgroups with
cyclic quotients. Every finitely generated nilpotent group is supersolvable~\cite{robinson:book96}, and the
symmetric group $Sym(4)$ is the simplest example of a solvable but not supersolvable group.

\begin{thm} \label{polyciclic}
The Gap Conjecture holds for residually supersolvable  groups. Moreover,  the conjecture $G^*(1/2)$ holds for
residually supersolvable groups
\end{thm}

Developing his technique and using the  known facts about maximal primitive solvable subgroups of $GL_n(p)$
($p$ prime)  Wilson in \cite{wilson:gap11} proved that the  Gap Conjecture with parameter $1/6$ holds for
residually solvable groups. In fact  what  follows from arguments in~\cite{wilson:growthsolv05}, combined
with arguments from\cite{grigorch:hilbert,lubotzky_mann:polyn91} and  with what was written above, can be
formulated as

\begin{thm} \label{wilson6}  The conjecture $G^*(1/6)$ holds for residually solvable groups.

\end{thm}

There is a hope that eventually  the Gap Conjecture and its $*$-version will be proved for residually
solvable groups, or at least for residually polycyclic groups (which is the same as to prove it for  groups
approximated by finite solvable groups, because polycyclic groups are residually finite
\cite{robinson:book96}). If the latter is done, then we will have  complete reduction of the Gap Conjecture
to just-infinite groups (more on this in the last section).

\section{Gap Conjecture for right orderable groups}

Recall that a group is called right orderable if there is a linear order on the set of its elements invariant
with respect to multiplication on the right. In a similar way are defined left orderable groups. A group is
bi-orderable (or totally orderable) if there is a linear order invariant with respect to multiplication on
the left and on the right. Every right orderable group is left orderable and vise versa but there are right
orderable groups which are not totally orderable (see  \cite{kokorin_kopyt:fully74} for examples). As was
shown by A.~Machi and  the author the class of finitely generated right orderable groups of intermediate
growth is nonempty \cite{grigorch_machi:93}.   The corresponding group $\hat{\mathcal G}$ was earlier
constructed in \cite{grigorch:example} as an example of a torsion free group of intermediate growth. It was
implicitly observed in \cite{grigorch_machi:93} that the class of countable right orderable groups coincides
with the class of groups acting faithfully by homeomorphisms on the line $\mathbb R$ (or, what is the same,
on the interval $[0,1]$).    Recently A.~Erschler and L.~Bartholdi managed to compute the growth of
$\hat{\mathcal G}$ which happens to be $e^{\log(n)n^{\alpha_0}}$  where $\alpha_0=\log 2/\log (2/\rho)
\approx 0.7674$, and $\rho$ is the real root of the polynomial
 $x^3+x^2+x-2$. The question
if there exists a finitely generated, totally orderable group of intermediate growth  is still open.

The Gap Conjecture and it modifications stated in section \ref{conjecture} are interesting problems even for
the class of right orderable groups. Our next result makes some contribution to this topic. The result of
Wilson combined with theorems of Morris~\cite{morris:amenable_gps_acting_on_line} and Rosset~\cite{rosset:76}
can be used to prove the following statement.

\begin{thm}\label{orderable}
(i) The Gap Conjecture  with parameter $1/6$,   and, moreover, the conjecture $G^*(1/6)$ hold for right
orderable groups.

(ii) The Gap Conjecture  $G(1/2)$ (or its $*$-version $G^*(1/2)$) holds for right orderable groups if it (or
its $*$-version $G^*(1/2)$) holds for residually polycyclic groups.
\end{thm}

\begin{proof}

(i)   Let $G$ be a finitely generated right orderable group with growth $\prec e^{n^{1/6}}$.   In
\cite{morris:amenable_gps_acting_on_line} D.~Morris proved that every finitely generated right orderable
amenable group is  indicable (i.e. can be mapped onto $\mathbb Z$). As by Adelson-Velskii theorem
\cite{adelson:banach57} a group of intermediate growth is amenable, we conclude  that the abelianization
$G_{ab}=G/[G,G]$ is infinite and hence has a decomposition $G_{ab}=G^{-}_{ab}\oplus G^{+}_{ab}$ where
$G^{-}_{ab}\simeq \mathbb Z ^d, d \geq 1$ is a torsion free part of an abelian group and $G^{+}_{ab}$ is a
torsion part.   Let $N \lhd G$  be a normal subgroup such that $G/N=G^{-}_{ab}$. Since the commutator
subgroup of a group is a characteristic group and the torsion free part of  abelian group also is a
characteristic subgroup we conclude that $N$ is a characteristic subgroup of $G$.
 By Theorem~\ref{rosset}  $N$ is a finitely generated group.  Therefore we can proceed with $N$ as we did
 with $G$.  This allows us to get a descending chain

 \begin{equation} \label{chain} G > G_1>G_2> \dots
 \end{equation}
(where $G_1=N$ etc)  of
 characteristic subgroups with the property that  $G_i/G_{i+1}\simeq \mathbb Z^{d_i}$ if $G_{i+1} \neq
 \{1\}$,
 for some sequence $d_i \in \mathbb N, i=1,2,\dots$.

 If the chain (\ref{chain}) terminates after finitely many steps then  $G$ is solvable and by the results of Milnor and Wolf
 \cite{milnor:solv68,wolf:growth68} $G$ is
 virtually nilpotent in this case.

  Suppose that chain (\ref{chain}) is infinite and consider the intersection $G_{\omega}=\bigcap_{i=1}^{\infty}
 G_i$.
 If $G_{\omega}=\{1\}$, then the group $G$ is residually solvable (in fact residually polycyclic), and, because
 of restriction on growth, by    Theorem~\ref{wilson6}, $G$ is virtually nilpotent and hence has polynomial
 growth of some degree $d$.  But this
contradicts   Splitting Lemma  \ref{splitting}. Therefore $G_{\omega} \neq \{1\}$. $G/G_{\omega}$ is
residually polycyclic, has growth not greater than the growth of $G$ and by previous argument is virtually
nilpotent. If the degree of polynomial growth of $G/G_{\omega}$ is $l$ then again by Splitting Lemma the
length of the chain (\ref{chain}) can not be larger than $l$, and we  get a contradiction. The part (i) of
the theorem is proven.

Now the proof of part  (ii) follows immediately.  If we assume  that $G$ has growth $\prec e^{\sqrt{n}}$ and
that the Gap Conjecture holds for the class of residually polycyclic groups then the arguments from previous
part (i) are applicable in the same manner. The only difference is that  instead of  Theorem~\ref{wilson6}
one should use the assumption that the Gap Conjecture holds for residually polyciclic groups.  The same
argument works in the case of conjecture $G^*(1/2)$.

\end{proof}

\section{Gap Conjecture  and just-infinite groups}

  There is a strong  evidence  based on considerations presented  below that the Gap Conjecture
  can be reduced to   three classes of
groups: \emph{simple} groups, \emph{branch} groups and  \emph{hereditary just-infinite}  groups. These three
types of groups appear in a natural partition  of the class of just-infinite groups
  into three subclasses described in Theorem~\ref{just-inf}.
 The following statement is an
easy application of Zorn's lemma.

\begin{prop} \label{quotient}
Let $G$ be a finitely generated infinite group.  Then $G$ has a just-infinite quotient.
\end{prop}

\begin{cor}  Let $\mathcal{P}$ be a group theoretical property preserved under taking quotients. If there
is a finitely generated  group satisfying the property  $\mathcal{P}$  then there is a just-infinite group
satisfying this property.
\end{cor}

Although the property of a group to have intermediate growth is not preserved when passing to a quotient
group (the image may have  polynomial growth), by  theorems of Gromov~\cite{gromov:poly_growth} and Rosset
\cite{rosset:76}, if the quotient $G/H$ of a group $G$ of intermediate growth has polynomial growth then $H$
is a finitely generated group (of intermediate growth, as the extension of a virtually nilpotent group by a
virtually nilpotent group is an elementary amenable group and therefore can not have intermediate growth),
and one may look for a just-infinite quotient of $H$ and iterate this process in order to represent $G$ as a
consecutive extension of a chain of groups that are virtually nilpotent or just-infinite groups. This
observation  was used in the previous section for the proof of Theorem~\ref{orderable} and is the base of the
arguments for Theorems~\ref{orderable},~\ref{polyciclic},~\ref{resid} and ~\ref{just-inf}.

Recall that hereditary just-infinite groups were already defined in section \ref{amen}.  We call a just
infinite group near simple if  it contains a subgroup of finite index which is a direct product of finitely
many copies of a simple group.

   Branch groups   are groups that have a faithful level transitive action on an infinite spherically homogeneous rooted tree $T_{\bar m}$ defined by a sequence
   $\{m_n\}_{n=1}^{\infty}$ of natural numbers $m_n \geq 2$  (determining the branching number  for vertices of level
   $n$) with the property that the rigid stabilizer $rist_G(n)$ has finite index in $G$ for each $n\geq 1$.
   Here by  $rist_G(n)$ we mean a subgroup $\prod_{v \in V_n} rist_G(v_n)$  which is a product of rigid
   stabilizers $rist_G(v_n)$ of vertices $v_n$ taken over the set $V_n$ of all vertices of level $n$, and
    $rist_G(v)$ is a subgroup of $G$ consisting of elements fixing the vertex $v$ and acting trivially
   outside the full subtree with the root at $v$.   For a more detailed discussion of this notion see
   \cite{grigorch:jibranch,bar_gs:branch}.
    This is a geometric definition.  It follows immediately from the definition that branch groups are infinite.
 The definition of an algebraically branch group can be found in~\cite{grigorch:branch,bar_gs:branch}. Every
geometrically branch group is algebraically branch but not vice versa. If $G$ is algebraically branch  then
it has a quotient $G/N$ which is  geometrically branch. The difference between  two versions of the
definitions is not large but still there is no complete understanding how much the two classes differ (it is
not clear what can be said about the  kernel $N$,  it is believed that it should be central in $G$). For
just-infinite branch groups the algebraic and geometric definitions are equivalent. Not every branch group is
just-infinite, but every proper quotient of a branch group is virtually abelian \cite{grigorch:jibranch}.
Therefore branch groups are ``almost just-infinite'' and most of known finitely generated branch groups are
just-infinite.   Observe that a finitely generated virtually nilpotent group is not branch. This follows for
instance from the fact that a finitely generated nilpotent group satisfies a minimal condition for normal
subgroups while a branch group not.

The next theorem  was derived by the author from a result of Wilson~\cite{wilson:ji71}.

\begin{thm}\label{just-inf}~\cite{grigorch:jibranch}
The class of just-infinite groups naturally splits into three subclasses: (B) branch just-infinite groups,
(H)  hereditary just-infinite groups, and (S) near-simple just-infinite groups.
\end{thm}

It is already known that  there  are  finitely generated  branch  groups of intermediate growth. For
 instance,  groups $\mathcal G_{\omega}$ of intermediate growth from the articles \cite{grigorch:continuum84,grigorch:degrees85} are of this type.
 In fact, all known examples of groups of intermediate growth are of branch type or are reconstructions on the
 base of groups of branch type.
The question about  existence of amenable but non-elementary amenable hereditary just-infinite group is still
open (remind that by Theorem~\ref{element} the only elementary amenable hereditary just-infinite groups are
$\mathbb{Z}$ and $D_{\infty}$).

\begin{prob}\label{medyn}
Are there  finitely generated hereditary just-infinite groups of intermediate growth?
\end{prob}

\begin{prob}
Are there  finitely generated  simple groups of intermediate growth?
\end{prob}

The next  theorem is a â  straightforward corollary of the main result of Bajorska and Makedonska from
\cite{bajorska_maked:note07} (observe that it was not stated in \cite{bajorska_maked:note07}). Here we
suggest a different   proof which is adapted to the needs of the proof of the main Theorem~\ref{just-inf1}.

\begin{thm}\label{resid}
If  the Gap Conjecture or conjecture $G^*(1/2)$ holds for the classes of  residually finite groups and simple
groups, then the corresponding conjecture holds for the class of all groups.
\end{thm}

\begin{proof}  Assume that the Gap Conjecture is correct for residually finite groups and for simple groups.
Let $G$ be a finitely generated group with growth $\prec e^{\sqrt n}$.  By Proposition \ref{quotient} it has
just-infinite quotient $\bar G =G/N$, which belongs to one of the three types of groups listed in the
statement of the Theorem~\ref{just-inf}.   The rate of growth of $\bar G$ is  not greater than the rate of
growth of $e^{\sqrt n}$.  The group $\bar G$  can not be near simple because in this case it will have a
subgroup $H$ of finite index with infinite finitely generated simple quotient  whose rate of growth is $\prec
e^{\sqrt n}$. This is impossible as a virtually nilpotent group can not be infinite simple.

The group $\bar G$   also can not be branch as branch groups are residually finite and  finitely generated
virtually nilpotent groups are not branch. So we can assume that $\bar G$  is hereditary just infinite and
hence residually finite. Using  the assumption of the theorem we conclude that $\bar G$  is virtually
nilpotent, and therefore elementary amenable. By Theorem~\ref{element} $\bar G$ is isomorphic either to the
infinite cyclic group or to the infinite dihedral group $D_{\infty}$.   By Theorem~\ref{rosset}  kernel $N$
is finitely generated. As the rate of growth of $N$ is less than $e^{\sqrt n}$ we can apply to $N$ the same
arguments as for $G$ in order to get a surjective homomorphism either onto $\mathbb Z$ or onto $D_{\infty}$.

 If $G/N\simeq \mathbb Z$,  then we
 repeat the first step of the proof of Theorem~\ref{orderable} replacing $N$ by a finitely generated
characteristic subgroup $N_1\lhd G$ with quotient $G/N_1 \simeq \mathbb Z ^{d_1}$ for some $d_1 \geq 1$. If
$G/N_1 \simeq D_{\infty}$ then  we slightly modify the first step.  Namely, in this case $G$ has  indicable
subgroup $H$ of index 2.   Let $H_1$ be the intersection of groups $H^{\phi}, \phi \in Aut(G)$. As there are
only finitely many subgroups of index 2 in $G$ this intersection   involves only finitely many groups and
$H_1$ is a characteristic subgroup in $G$ of  finite index of   type $2^t$ for some $ t \in \mathbb N$.
Moreover, $G/H_1\simeq \mathbb Z_2^{t}$  as the quotient $G/H_1$ is isomorphic to a subgroup of a direct
product of finitely many copies of group $\mathbb Z_2$ of order 2.  The subgroup $H_1$, being a subgroup of
index $2^{t-1}$ in $H$, is indicable and we can apply  the argument of the first step of the proof of
Theorem~\ref{orderable} getting a finitely generated  subgroup $H_2 \unlhd H_1$ characteristic in $G$ with
quotient $H_1/H_2\simeq \mathbb Z ^{d_1}$ for some $d_1 \in \mathbb N$.

Let  $G_1\lhd G$ be a subgroup $N$,$H_1$ or $H_2$ depending on the case. Proceed with $G_1 $ in a similar
fashion as we did with $G$, etc.  We get a descending chain $\{G_i\}_{i\geq 1}$ of finitely generated
subgroups characteristic in $G$. There are two possibilities.

1)  After finitely many steps we  get a group $G_i$ which is hereditary just-infinite and  elementary
amenable, and  hence infinite cyclic or $D_{\infty}$ (Theorem~\ref{element}). In this case $G$ is polycyclic
and we are done in view of the result of Milnor and Wolf on growth of solvable groups.

2) The process of construction of the chain of subgroups will continue forever. In this case we get a chain
with the property that $G_i/G_{i+1}$ is isomorphic either to (i) $\mathbb Z ^{d_i}, d_i \in \mathbb N$ or to
(ii) $\mathbb Z_2^{t_i}, t_i \in \mathbb N$. Moreover,  each step of type (ii) is immediately followed by a
step of type (i).

%The end of the proof basically is the same  as the end of the proof of Theorem~\ref{orderable}.

  Let us show that this is impossible. Let $G_{\omega}$ be the intersection
$\bigcap_{i\geq 1} G_i$.  Then $G/G_{\omega}$ is residually polycyclic and hence residually finite as every
polycyclic group is residually finite \cite{robinson:book96}. Growth of $G/G_{\omega}$ is less than $e^{\sqrt
n}$. Hence by the assumption of the theorem the group $G/G_{\omega}$ is virtually nilpotent with the rate of
polynomial growth of degree $d$ for some $d \in \mathbb N$. But this contradicts  the splitting lemma as for
infinitely many $i$ the quotients $G_i/G_{i+1}$ are isomorphic to $\mathbb Z^{d_i}$. This proves the
conjecture $G(1/2)$.

In the case of the conjecture $G^*(1/2)$ we proceed in a similar fashion.  Only at the beginning we assume
that the conjecture $G^*(1/2)$   holds for residually finite groups and for simple groups and that $G$ is a
finitely generated group of intermediate  growth whose growth does not satisfy inequality  $\gamma(n) \succeq
e^{n^{1/2}}$.

\end{proof}

Now we state and prove our main result.

\begin{thm}\label{just-inf1}
(i) If  the Gap Conjecture  with parameter $1/6$ or its $*$-version $G^*(1/6)$ holds for just-infinite groups
then the corresponding conjecture holds for all groups.

(ii) If  the Gap Conjecture  or its $*$-version $G^*(1/2)$ holds for residually polycyclic groups and for
just-infinite groups then the corresponding conjecture holds for all groups.
\end{thm}

\begin{proof}  (i)  The proof follows the same strategy as the proof of Theorem~\ref{resid}. Let $G$ be a  finitely generated  group with growth
$\prec e^{ n^{1/6}}$. There can be  two possibilities.

1) $G$ has a finite  descending chain $\{G_i\}_{i=1}^k$ of finitely generated characteristic in $G$ groups
with consecutive quotients $G_i/G_{i+1} \simeq \mathbb Z^{d_i}$ or $G_i/G_{i+1} \simeq \mathbb Z_2^{t_i}$,
for $i <k$ and $G_k=\{1\}$. In this case $G$ is polycyclic and hence virtually nilpotent

2)   $G$ has an infinite  descending chain $\{G_i\}_{i=1}^{\infty}$,  with the property that $G_i/G_{i+1}
\simeq \mathbb Z^{d_i}$ or $G_i/G_{i+1} \simeq \mathbb Z_2^{t_i}$,  and if $G_i/G_{i+1} \simeq \mathbb
Z_2^{t_i}$ then $G_{i+1}/G_{i+2} \simeq \mathbb Z^{d_{i+1}}$.  The group $G/G_{\omega}$, where
$G_{\omega}=\bigcap_{i\geq 1} G_i$, is residually polycyclic with growth $\prec e^{ n^{1/6}}$. Apply in this
case the result of Wilson stated in Theorem~\ref{wilson2} concluding that $G/G_{\omega}$  is virtually
nilpotent which is impossible by the splitting lemma.

(ii)  Proceed as in (i) with the only difference  that in the subcase 2) we apply the assumption that the Gap
Conjecture holds for residually polycyclic groups
% only in the case of infinite chain $\{G_i\}_{i\geq
%1}$ apply hypothesis of the part (ii) of the theorem to $G/G_{\omega}$
to conclude that this subcase is impossible.

These are arguments for $G(1/2)$ version.   The arguments for $*$-version $G^*(1/2)$   are similar.

\end{proof}

\bibliographystyle{plain}
\bibliography{mylib}

\def\cprime{$'$} \def\cprime{$'$} \def\cprime{$'$} \def\cprime{$'$}
  \def\cprime{$'$} \def\cprime{$'$} \def\cprime{$'$} \def\cprime{$'$}
  \def\cprime{$'$}
\begin{thebibliography}{10}

\bibitem{adelson:banach57}
G.~M. Adel{\cprime}son-Vel{\cprime}ski{\u\i} and Yu.~A. {\v{S}}re{\u\i}der.
\newblock The {B}anach mean on groups.
\newblock {\em Uspehi Mat. Nauk (N.S.)}, 12(6(78)):131--136, 1957.

\bibitem{bajorska_maked:note07}
B.~Bajorska and O.~Macedo{\'n}ska.
\newblock A note on groups of intermediate growth.
\newblock {\em Comm. Algebra}, 35(12):4112--4115, 2007.

\bibitem{bartholdi_ersch:given(2)11}
L.~Bartholdi and A.~Erschler.
\newblock Groups of given intermediate word growth, 2011.
\newblock (available at \emph{http://arxiv.org/abs/1110.3650}).

\bibitem{bar_gs:branch}
Laurent Bartholdi, Rostislav~I. Grigorchuk, and Zoran {\v{S}}uni{\'k}.
\newblock Branch groups.
\newblock In {\em Handbook of algebra, Vol. 3}, pages 989--1112. North-Holland,
  Amsterdam, 2003.

\bibitem{chou:eg}
Ching Chou.
\newblock Elementary amenable groups.
\newblock {\em Illinois J. Math.}, 24(3):396--407, 1980.

\bibitem{day:amenable}
Mahlon~M. Day.
\newblock Amenable semigroups.
\newblock {\em Illinois J. Math.}, 1:509--544, 1957.

\bibitem{harpe_cg:paradoxical}
P.~de~la Harpe, R.~I. Grigorchuk, and T.~Ceccherini-Silberstein.
\newblock Amenability and paradoxical decompositions for pseudogroups and
  discrete metric spaces.
\newblock {\em Tr. Mat. Inst. Steklova}, 224(Algebra. Topol. Differ. Uravn. i
  ikh Prilozh.):68--111, 1999.

\bibitem{harpe:topics}
Pierre de~la Harpe.
\newblock {\em Topics in geometric group theory}.
\newblock Chicago Lectures in Mathematics. University of Chicago Press,
  Chicago, IL, 2000.

\bibitem{freedman_teich:subexp95}
Michael~H. Freedman and Peter Teichner.
\newblock {$4$}-manifold topology. {I}. {S}ubexponential groups.
\newblock {\em Invent. Math.}, 122(3):509--529, 1995.

\bibitem{grigorch:milnor}
R.~I. Grigorchuk.
\newblock On the {M}ilnor problem of group growth.
\newblock {\em Dokl. Akad. Nauk SSSR}, 271(1):30--33, 1983.

\bibitem{grigorch:continuum84}
R.~I. Grigorchuk.
\newblock Construction of {$p$}-groups of intermediate growth that have a
  continuum of factor-groups.
\newblock {\em Algebra i Logika}, 23(4):383--394, 478, 1984.

\bibitem{grigorch:degrees}
R.~I. Grigorchuk.
\newblock Degrees of growth of finitely generated groups and the theory of
  invariant means.
\newblock {\em Izv. Akad. Nauk SSSR Ser. Mat.}, 48(5):939--985, 1984.

\bibitem{grigorch:degrees85}
R.~I. Grigorchuk.
\newblock Degrees of growth of {$p$}-groups and torsion-free groups.
\newblock {\em Mat. Sb. (N.S.)}, 126(168)(2):194--214, 286, 1985.

\bibitem{grigorch:semigroups_with_cancel88}
R.~I. Grigorchuk.
\newblock Semigroups with cancellations of polynomial growth.
\newblock {\em Mat. Zametki}, 43(3):305--319, 428, 1988.

\bibitem{grigorch:hilbert}
R.~I. Grigorchuk.
\newblock On the {H}ilbert-{P}oincar\'e series of graded algebras that are
  associated with groups.
\newblock {\em Mat. Sb.}, 180(2):207--225, 304, 1989.

\bibitem{grigorch:example}
R.~I. Grigorchuk.
\newblock An example of a finitely presented amenable group that does not
  belong to the class {EG}.
\newblock {\em Mat. Sb.}, 189(1):79--100, 1998.

\bibitem{grigorch:branch}
R.~I. Grigorchuk.
\newblock Branch groups.
\newblock {\em Mat. Zametki}, 67(6):852--858, 2000.

\bibitem{grigorch:jibranch}
R.~I. Grigorchuk.
\newblock Just infinite branch groups.
\newblock In {\em New horizons in pro-$p$ groups}, volume 184 of {\em Progr.
  Math.}, pages 121--179. Birkh\"auser Boston, Boston, MA, 2000.

\bibitem{grigorch_machi:93}
R.~I. Grigorchuk and A.~Machi.
\newblock On a group of intermediate growth that acts on a line by
  homeomorphisms.
\newblock {\em Mat. Zametki}, 53(2):46--63, 1993.

\bibitem{grigorch:habil}
R.I. Grigorchuk.
\newblock {\em Groups with intermediate growth function and their
  applications}.
\newblock Habilitation, Steklov Institute of Mathematics, 1985.

\bibitem{grigorch:gap2}
Rostislav Grigorchuk.
\newblock The gap type conjectures for various asymptotic characteristics of
  groups.
\newblock In preparation.

\bibitem{grigorch:milnor11}
Rostislav Grigorchuk.
\newblock Milnor's problem on the growth of groups and its consequences, 2011.
\newblock (available at \emph{http://arxiv.org/abs/1111.0512}).

\bibitem{grigorch:ICM90}
Rostislav~I. Grigorchuk.
\newblock On growth in group theory.
\newblock In {\em Proceedings of the {I}nternational {C}ongress of
  {M}athematicians, {V}ol.\ {I}, {II} ({K}yoto, 1990)}, pages 325--338, Tokyo,
  1991. Math. Soc. Japan.

\bibitem{gromov:poly_growth}
Mikhael Gromov.
\newblock Groups of polynomial growth and expanding maps.
\newblock {\em Inst. Hautes \'Etudes Sci. Publ. Math.}, (53):53--73, 1981.

\bibitem{karpilovsky:schur}
Gregory Karpilovsky.
\newblock {\em The {S}chur multiplier}, volume~2 of {\em London Mathematical
  Society Monographs. New Series}.
\newblock The Clarendon Press Oxford University Press, New York, 1987.

\bibitem{kokorin_kopyt:fully74}
Ali~Ivanovi{\v{c}} Kokorin and Valeri{\={\i}}~Matveevi{\v{c}} Kopytov.
\newblock {\em Fully ordered groups}.
\newblock Halsted Press [John Wiley\thinspace \&\thinspace Sons], New
  York-Toronto, Ont., 1974.
\newblock Translated from the Russian by D. Louvish.

\bibitem{rhemtulla_long:freesemig95}
P.~Longobardi, M.~Maj, and A.~H. Rhemtulla.
\newblock Groups with no free subsemigroups.
\newblock {\em Trans. Amer. Math. Soc.}, 347(4):1419--1427, 1995.

\bibitem{lubotzky_mann:polyn91}
Alexander Lubotzky and Avinoam Mann.
\newblock On groups of polynomial subgroup growth.
\newblock {\em Invent. Math.}, 104(3):521--533, 1991.

\bibitem{milnor:note68}
J.~Milnor.
\newblock A note on curvature and fundamental group.
\newblock {\em J. Differential Geometry}, 2:1--7, 1968.

\bibitem{milnor:problem}
J.~Milnor.
\newblock Problem $5603$.
\newblock {\em Amer. Math. Monthly}, 75:685--686, 1968.

\bibitem{milnor:solv68}
John Milnor.
\newblock Growth of finitely generated solvable groups.
\newblock {\em J. Differential Geometry}, 2:447--449, 1968.

\bibitem{morris:amenable_gps_acting_on_line}
Dave~Witte Morris.
\newblock Amenable groups that act on the line.
\newblock {\em Algebr. Geom. Topol.}, 6:2509--2518, 2006.

\bibitem{osin:elementary04}
D.~V. Osin.
\newblock Algebraic entropy of elementary amenable groups.
\newblock {\em Geom. Dedicata}, 107:133--151, 2004.

\bibitem{robinson:book96}
Derek J.~S. Robinson.
\newblock {\em A course in the theory of groups}, volume~80 of {\em Graduate
  Texts in Mathematics}.
\newblock Springer-Verlag, New York, second edition, 1996.

\bibitem{rosset:76}
Shmuel Rosset.
\newblock A property of groups of non-exponential growth.
\newblock {\em Proc. Amer. Math. Soc.}, 54:24--26, 1976.

\bibitem{shalom_tao:polynom10}
Yehuda Shalom and Terence Tao.
\newblock A finitary version of {G}romov's polynomial growth theorem.
\newblock {\em Geom. Funct. Anal.}, 20(6):1502--1547, 2010.

\bibitem{tits:alternative72}
J.~Tits.
\newblock Free subgroups in linear groups.
\newblock {\em J. Algebra}, 20:250--270, 1972.

\bibitem{vonNeumann:1929}
John von Neumann.
\newblock Zurr allgemeinen theorie des masses.
\newblock {\em Fund.Math.}, 13:73--116, 1929.

\bibitem{schvarts:55}
Albert \v{S}varc.
\newblock A volume invariant of covering.
\newblock {\em Dokl. Akad. Nauj SSSR}, 105:32--34, 2001955.

\bibitem{wilson:gap11}
J.~Wilson.
\newblock The gap in the growth of residually soluble groups.
\newblock {\em Bull. London Math. Soc.}

\bibitem{wilson:ji71}
J.~S. Wilson.
\newblock Groups with every proper quotient finite.
\newblock {\em Proc. Cambridge Philos. Soc.}, 69:373--391, 1971.

\bibitem{wilson:growthsolv05}
John~S. Wilson.
\newblock On the growth of residually soluble groups.
\newblock {\em J. London Math. Soc. (2)}, 71(1):121--132, 2005.

\bibitem{wolf:growth68}
Joseph~A. Wolf.
\newblock Growth of finitely generated solvable groups and curvature of
  {R}iemanniann manifolds.
\newblock {\em J. Differential Geometry}, 2:421--446, 1968.

\end{thebibliography}

\end{document}